\documentclass{article}
\usepackage{amsfonts}
\usepackage{amsmath}
\usepackage{xypic}

\newenvironment{proof}[1][Proof]{\textbf{#1 }}{$\square$}
\newtheorem{theorem}{Theorem}
\newtheorem{acknowledgement}[theorem]{Acknowledgement}

\newtheorem{proposition}[theorem]{Proposition}

\begin{document}

\title{Para-Sasaki-like Manifolds with Generalized Symmetric Metric Connection}
\date{}
\maketitle
\begin{center}
\c{S}enay BULUT, Pınar İNSELÖZ\\[0pt]
Eskişehir Technical University, Department of Mathematics, Eskişehir, TURKEY\\[0pt]
\texttt{skarapazar@eskisehir.edu.tr, pinardeniz136@gmail.com}

\end{center}

\begin{abstract}
The present paper deals with the generalized symmetric metric connection defined on para-Sasaki-like manifolds. We derive a relation between the Levi-Civita connection and the generalized symmetric metric conneciton on the considered manifold. We investigate the curvature tensor, the Ricci tensor and scalar curvature tensor with respect to the generalized symmetric metric connection. We study para-Sasaki-like solitons on para-Sasaki-like manifolds with  the generalized symmetric metric connection. Finally, we construct two examples of para-Sasaki-like manifolds admitting generalized symmetric metric connection and verify our some results.

\end{abstract}


\textbf{2000 MR Subject Classification} 53C15, 53C25, 53C50. \label{first}

\section{Introduction}

The study of a semisymmetric linear connection on a differentiable manifold was initiated by Friedmann and Schouten \cite{Fri} in 1924. Later, in \cite{Hayden} the idea of semisymmetric linear connection with a torsion different from zero was introduced and studied. This connection has been studied on manifolds equipped with different structures such as almost contact B-metric manifolds, almost contact manifolds, Kenmotsu manifolds \cite{Sharfuddin, Tripati, bulut2}. A quarter-symmetric linear connection on a differentiable manifold was introduced by Golab \cite{Golab}. Rastogi \cite{Rastogi1, Rastogi2} started the systematic study of the quarter-symmetric metric connection. Some properties of semisymmetric metric connections have been investigated on manifolds having special structures \cite{Mondal, bulut1, Mishra, Yano}. A generalized symmetric metric connection which is the generalized form of the semisymmetric metric connection and the quarter-symmetric metric connection for Kenmotsu manifolds, $L_p$-Sasakian manifolds and Golden Lorentzian manifolds has been defined and studied \cite{Bahad3, Bahad2, Bahad1}.

The study of Riemannian Ricci solitons was carried out in \cite{bib4}. The investigations into the Ricci solitons have been generalized in contact geometry, paracontact geometry and pseudo-Riemannian geometry \cite{bib5,bib6,bib7}. A detailed study on para-Ricci-like soliton has been performed in \cite{bib2}.

The geometry of almost paracontact almost paracomplex Riemannian manifolds (briefly apapR manifolds) has been examined \cite{bib3}.  After that, para-Sasaki-like manifolds, a special class of apapR manifolds, have been presented in \cite{bib2}. Para-Ricci-like solitons with Reeb vector field potential, arbitrary potential and vertical potential have been investigated and proved some additional geometric properties in \cite{bib1, bib2, bib8}.

Motivated by these circumstances we initiate the study of para-Sasaki-like manifolds with generalized symmetric metric connection which is the generalization of semi-symmetric and quarter-symmetric metric connection. In Section 2, we present some basic definitions and results about para-Sasaki-like manifolds. In Section 3, we define the generalized symmetric metric connection on apapR manifolds and calculate curvature tensor and the Ricci tensor with respect to this connection for para-Sasaki-like Riemannian manifolds. Moreover, we investigate some operators with this connection. In Section 4, we prove that the cases of para-Einstein-like of the manifold and admitting para-Ricci-like soliton are preserved in this connection. Finally, we give two examples to illustrate the obtained results.

\section{Para-Sasaki-like Riemannian Manifolds}
Let $M$ be a differentiable manifold of dimension $(2n+1)$ endowed with a tensor field $\phi$  of type $(1,1)$, a vector field $\xi$, a $1-$form $\eta$ and a Riemannian metric $g$, which satisfies 
\begin{equation}
\label{1}
\begin{array}{l}
\eta(\xi)=1, \ \ \  \phi^2x=x-\eta(x)\xi,  \ \ \ \phi \xi=0\\
g(\phi x,\phi y)=g(x,y)-\eta(x)\eta(y), \ \ \ tr \phi=0
    \end{array}
\end{equation}
for any $x,y\in \chi(M)$. Such manifold $(M,\phi,\xi,\eta,g)$ is called an almost paracontact almost paracomplex Riemannian manifold (shortly, apapR manifold)\cite{bib3}.
Moreover, by using the latter equalities we get the following:
\begin{equation}
\label{2}
\begin{array}{lll}
  g(x,\xi)=\eta(x),  &  &g(x,\phi y)=g(\phi x,y),  \\
    g(\xi,\xi)=1,  &  & \eta(\nabla_x\xi)=0,
    \end{array}
\end{equation}
where $\nabla$ denotes the Levi-Civita connection of $g$.

The associated metric $\widetilde{g}$ of $g$ on $(M,\phi,\xi,\eta,g)$ determined by the equality $\widetilde{g}(x,y)=g(x,\phi y)+\eta(x)\eta(y)$ is a pseudo-Riemannian metric of signature $(n+1,n)$.

An apapR manifold $M$ is said to be a para-Sasaki-like manifold if the following is provided:
\begin{equation}
\label{3}
\begin{array}{lll}
  (\nabla_x \phi )y &=  & -g(x,y)\xi-\eta(y)x+2\eta(x)\eta(y)\xi,   \\
      &  = &-g(\phi x,\phi y)\xi-\eta(y)\phi^2 x.
\end{array}
\end{equation}
In \cite{Ivanov}, it is proven that the following identities hold for any para-Sasaki-like manifold $(M,g,\phi,\xi,\eta)$:
\begin{equation}
\label{4}
\begin{array}{lll}
\nabla_x \xi=\phi x, &   &(\nabla_x\eta)y=g(x,\phi y), \\
R(x,y)\xi=-\eta(y)x+\eta(x)y,&   &\text{Ric}(x,\xi)=-2n\ \eta(x),\\
R(\xi,y)\xi=\phi^2 y,  &&  \text{Ric}(x,\xi)(\xi,\xi)=-2n,
\end{array}
\end{equation}
where $R$ and $\text{Ric}$ denote the curvature tensor and the Ricci tensor, respectively.
If $M$ is para-Sasaki-like Riemannian manifold, then we have the following identities:
\begin{equation}
\label{5}
\begin{array}{lll}
\phi\nabla_x\xi=\nabla_{\phi x}\xi=x-\eta(x)\xi, \\
\end{array}
\end{equation}

\section{ Para-Sasaki-like Riemannian Manifolds with Generalized Symmetric Metric Connection }
In this section we consider a generalized symmetric metric connection on an apapR manifold. 

Let $(M,\phi,\xi,\eta,g)$ be an apapR manifold. We define a generalized symmetric metric connection $\overline{\nabla}$ on $M$ by
\begin{equation}
\label{15}
\overline{\nabla}_xy=\nabla_xy+\alpha\{g(x,y)\xi-\eta(y)x\}+\beta\{g(\varphi x,y)\xi-\eta(y)\varphi x\},
\end{equation}
where $\nabla$ is a Levi-Civita connection of $g$. By the latter equality the torsion tensor $T$ of $M$ corresponding to the connection $\overline{\nabla}$ is given by 
\begin{equation}\label{1100}
\begin{array}{lll}
T(x,y) & = & \overline{\nabla}_xy-\overline{\nabla}_yx-[x,y] \\
& = & \alpha\{\eta(x)y-\eta(y)x\}+\beta\{\eta(x)\varphi y-\eta(y)\varphi x\},
\end{array}
\end{equation}
for any vector field $x,y$ on $M$, where $\alpha$ and $\beta$ are smooth functions. The connection $ \overline{\nabla}$ is said to be generalized symmetric connection. If $(\alpha,\beta)=(1,0)$ or $(\alpha,\beta)=(0,1)$, then the  generalized symmetric connection is reduced to a semi-symmetric connection and quarter-symmetric connection, respectively. If there is a Riemannian metric $g$ on $M$ such that the connection $\overline{\nabla}$ satisfies the condition 
\begin{equation}
\label{1110}
(\overline{\nabla}_xg)(y,z)=0,
\end{equation}
for all $x,y,z\in \chi(M)$, then $\overline{\nabla}$ is called a generalized symmetric metric connection, otherwise it is called a generalized symmetric nonmetric connection.

Now we will show the existence of the generalized symmetric metric connection $\overline{\nabla}$ on an apapR manifold $(M,\phi,\xi,\eta,g)$.
\begin{theorem}
There exists a unique linear connection $\overline{\nabla}$ satisfying (\ref{1100}) and (\ref{1110}) on an apapR manifold $(M,\phi,\xi,\eta,g)$.
\end{theorem}
\begin{proof}
Suppose that $\overline{\nabla}$ is a linear connection determined by
\begin{equation}
\label{1101}
\overline{\nabla}_xy=\nabla_xy+H(x,y),
\end{equation}
for any vector field $x$ and $y$, where $H$ is a tensor of type $(1,2)$. Now, we will specify the  tensor field $H$ such that $\overline{\nabla}$ satisfies the relations (\ref{1100}) and (\ref{1110}). By using the definition of torsion tensor and (\ref{1101}), we obtain 
\begin{equation}
\label{13}
T(x,y)=H(x,y)-H(y,x),
\end{equation}
for all $x,y\in \chi(M)$. The latter equality leads to 
\begin{equation}
\label{1102}
g(T(x,y),z)=g(H(x,y),z)-g(H(y,x),z).
\end{equation}
The equation (\ref{1110}) yields to
\begin{equation*}
0=xg(y,z)-g(\overline{\nabla}_xy,z)-g(\overline{\nabla}_xz,y)=g(H(x,y),z)+g(H(x,z),y).
\end{equation*}
Then, the latter equality leads to 
\begin{equation}
\label{1103}
g(H(x,y),z)=-g(H(x,z),y).
\end{equation}
By virtue of (\ref{1102}) and (\ref{1103}) we can easily calculate the following:
\begin{equation}
\label{1105}
g(T(x,y),z)+g(T(z,x),y)+g(T(z,y),x)=2g(H(x,y),z).
\end{equation}
By using (\ref{1100}) and (\ref{1105}) we get 
\begin{equation}
\label{1106}
\begin{array}{lll}
  2g(H(x,y),z)& =& g(T(x,y),z)+g(T(z,x),y)+g(T(z,y),x)   \\
      & =  & 2g(\alpha(g(x,y)\xi-\eta(y)x)+\beta(g(\varphi x,y)\xi-\eta(y)\varphi x),z).
\end{array}
\end{equation}
Hence, (\ref{1106}) implies 
\begin{equation}
\label{ }
H(x,y)= \alpha\{g(x,y)\xi-\eta(y)x\}+\beta\{g(\varphi x,y)\xi-\eta(y)\varphi x\}.
\end{equation}
The latter equality leads to Equation (\ref{15}). 
\end{proof}

By means of (\ref{3}), (\ref{4}) and (\ref{5}), we obtain the following proposition.
\begin{proposition}
Let $M$ be a para-Sasaki-like Riemannian manifold with the generalized symmetric metric connection. We have the following relations:
\begin{equation}
\label{ }
\begin{array}{lll}
  \overline{\nabla}_x \varphi&=    &(\beta-1) \{g(\varphi x,\varphi y)\xi+\eta(y)\varphi^2(x)\}+\alpha\{g(x,\varphi y)\xi+\eta(y)\varphi x\},  \\
   \overline{\nabla}_x \xi &=       &   (1-\beta)\varphi x-\alpha\varphi^2 x,\\
  ( \overline{\nabla}_x \eta) y&=       &   (1-\beta)g(\varphi x,y)-\alpha g(x,y)+\eta (x) \eta (y),\\   
\end{array}
\end{equation}
for any $x,y\in \chi(M)$.
\end{proposition}

\begin{proposition}
Let $(M,g,\phi,\xi,\eta)$ be a $(2n+1)-$dimensional para-Sasaki-like Riemannian manifold with the generalized symmetric metric connection. The relation between the curvature tensors of the generalized symmetric metric connection $ \overline{\nabla}$ and the Levi-Civita connection $\nabla$ is given by the following formulas:
  \begin{equation}
\label{ }
\begin{array}{lll}
 \overline{R}(x,y)z&=&R(x,y)z+x(\alpha) g(y,z)\xi+(\alpha-\alpha\beta)g(y,z)\varphi x  \\
      & &-x(\alpha)\eta(z)y+x(\beta)g(\varphi y,z)\xi +(2\beta-\beta^2)g(\varphi y,z)\varphi x\\
      &&-x(\beta)\eta(z)\varphi y+(\alpha^2+\beta)g(y,z)\eta(x)\xi+\alpha\beta g(\varphi y,z)\eta(x)\xi\\
      &&-\alpha^2g(y,z)x+(\alpha^2+\beta)\eta(y)\eta(z)x+(\alpha-\alpha\beta)g(\varphi y,z)x\\
      &&+\alpha\beta\eta(y)\eta(z)\varphi x-y(\alpha) g(x,z)\xi+y(\alpha)\eta(z)x\\
      &&-y(\beta)g(\varphi x,z)\xi+(\beta^2-2\beta)g(\varphi x,z)\varphi y+y(\beta)\eta(z)\varphi x\\
      &&-(\alpha^2+\beta)\eta(y)g(x,z)\xi-\alpha\beta g(\varphi x,z)\eta(y)\xi+\alpha^2 g(x,z)y\\
      &&-(\alpha^2+\beta)\eta(x)\eta(z)y+(\alpha\beta-\alpha)g(\varphi x,z)y\\
      &&-\alpha\beta\eta(x)\eta(z)\varphi y+(\alpha\beta-\alpha)g(x,z)\varphi y. \\
\end{array}
\end{equation}
\end{proposition}
\begin{proof}
The curvature tensor  $\overline{R}$ of the generalized symmetric metric connection $ \overline{\nabla}$ on $M$ is defined by
\begin{equation}
\label{ }
\overline{R}(x,y)z= \overline{\nabla}_x \overline{\nabla}_y z- \overline{\nabla}_y \overline{\nabla}_x z- \overline{\nabla}_{[x,y]}z.
\end{equation}
Substituting (\ref{15}) in the last equation and using Proposition \ref{1}, through (\ref{1}), (\ref{2}), (\ref{3}), we complete the proof by direct computation.
\end{proof}

If $\alpha,\beta$ are constant functions, then the curvature tensor $ \overline{R}$ is given by
  \begin{equation}
\label{8}
\begin{array}{lll}
 \overline{R}(x,y)z&=&R(x,y)z+(\alpha-\alpha\beta)g(y,z)\varphi x+(2\beta-\beta^2)g(\varphi y,z)\varphi x  \\
      &&+(\alpha^2+\beta)g(y,z)\eta(x)\xi+\alpha\beta g(\varphi y,z)\eta(x)\xi\\
      &&-\alpha^2g(y,z)x+(\alpha^2+\beta)\eta(y)\eta(z)x+(\alpha-\alpha\beta)g(\varphi y,z)x\\
      &&+\alpha\beta\eta(y)\eta(z)\varphi x+(\beta^2-2\beta)g(\varphi x,z)\varphi y\\
      &&-(\alpha^2+\beta)\eta(y)g(x,z)\xi-\alpha\beta g(\varphi x,z)\eta(y)\xi+\alpha^2 g(x,z)y\\
      &&-(\alpha^2+\beta)\eta(x)\eta(z)y+(\alpha\beta-\alpha)g(\varphi x,z)y\\
      &&+(\alpha\beta-\alpha)g(x,z)\varphi y-\alpha\beta \eta(x)\eta(z)\varphi y. \\
\end{array}
\end{equation}
Setting $z=\xi$ in the above equality and using (\ref{1}), (\ref{2}) and (\ref{4}), we get
  \begin{equation}
\label{100}
\begin{array}{lll}
 \overline{R}(x,y)\xi&=& \alpha\eta(y)\varphi x+(\beta-1)\eta(y)x-(\beta-1)\eta(x)y-\alpha \eta(x)\varphi y. \\
\end{array}
\end{equation}
 Plugging $x=\xi$ in (\ref{100}) and using (\ref{1}) we obtain the following equality:
  \begin{equation}
\label{9}
\begin{array}{lll}
 \overline{R}(\xi,y)\xi&=&(1-\beta)\varphi^2 y-\alpha \varphi(y). \\
\end{array}
\end{equation}

Let $\{\xi,e_1,\ldots,e_{2n}\}$ be an orthonormal basis of the tangent space at each point of the manifold $M$. If $\overline{R}(x,y,z,w)=g(\overline{R}(x,y)z,w)$, then the Ricci tensor $\overline{\text{Ric}}$ and the scalar curvature $\overline{\text{scal}}$ with respect to generalized symmetric metric connection $\overline{\nabla}$
are presented by 
\begin{equation*}
\label{21}
\overline{\mathrm{Ric}}(x,y)=\sum_{i=0}^{2n} \overline{R}(e_i,x,y,e_i)
\end{equation*}
and
\begin{equation}
\label{112}
\overline{\mathrm{scal}}=\sum_{i=0}^{2n} \overline{\mathrm{Ric}}(e_i,e_i),
\end{equation}
respectively.

\begin{theorem}
Let $(M,g,\phi,\xi,\eta)$ be a $(2n+1)-$dimensional para-Sasaki-like Riemannian manifold with the generalized symmetric metric connection. Then, there exists the following relationship between the Ricci curvatures $\mathrm{Ric}$ and $\overline{\mathrm{Ric}}$, respectively:

  \begin{equation}
\label{111}
\begin{array}{lll}
 \overline{\mathrm{Ric}}(x,y)&=&\mathrm{Ric}(x,y)+[\beta^2-\beta+(1-2n)\alpha^2]g(x,y)  \\
      &   &+[(2n+1)\beta-\beta^2+(2n-1)\alpha^2]\eta(x)\eta(y)\\
      && +[(\alpha\beta-\alpha)(2-2n)+\alpha] g(x,\varphi y) 
\end{array}
\end{equation}
\end{theorem}
\begin{proof}
Taking into account (\ref{21}) and (\ref{8}), we immediately obtain the Ricci tensor $ \overline{\mathrm{Ric}}$.
\end{proof}

Using (\ref{112}) and (\ref{111}), we obtain the following theorem.
\begin{theorem}
Let $(M,g,\phi,\xi,\eta)$ be a $(2n+1)-$dimensional para-Sasaki-like Riemannian manifold with the generalized symmetric metric connection. The scalar curvatures $\mathrm{scal}$ and $\overline{\mathrm{scal}}$ satisfy the following relation:
   \begin{equation}
\label{ }
\begin{array}{lll}
 \overline{\mathrm{scal}}&=&\mathrm{scal}+2n[\beta^2+(1-2n)\alpha^2].   \\
\end{array}
\end{equation}
\end{theorem}

\begin{theorem}
The Hessian, the divergence and the Laplace operators with respect to the generalized symmetric metric connection $\overline{\nabla}$ satisfy the following relations:
\begin{eqnarray}
\overline{\mathrm{div}}\ (x) & = & \mathrm{div}\ x-2n\alpha\ \eta(x)\label{113}\\
 \overline{\mathrm{Hess}} f (x,y)&=&\mathrm{Hess} f(x,y)+\alpha (xf)\eta(y)-\alpha(\xi f)g(x,y)\\
 &  &+\beta(\varphi x)f\eta(y)-\beta(\xi f)g(\varphi x,y)\\
 \overline{\Delta} f&=& \Delta f-2n\alpha\ \xi f
\end{eqnarray}
\end{theorem}
\begin{proof}
Let $\{\xi,e_1,\ldots,e_{2n}\}$ be an orthonormal basis of $T_pM$, $p\in M$. By taking the trace of $\overline{\nabla}x$, we get
$$ \overline{\mathrm{div}}\ (x)=\displaystyle \sum_{i=0}^{2n}g(\overline{\nabla}_{e_i}x,e_i).$$
The latter equality together with (\ref{15}) completes the first part of proof.

By definition we have $$ \overline{\mathrm{Hess}} f (x,y)=g(\overline{\nabla}_xgrad\ f,y).$$
We know that $g(grad\ f,x)=x(f)$. By virtue of  (\ref{15}), we prove the second part.

Substituting $x=grad\ f$ in (\ref{113}) and using (\ref{15}), we obtain the Laplace operator $ \overline{\Delta} f$.

\end{proof}

\section{Para-Ricci-like solitons on Para-Sasaki-like Riemannian Manifolds With Generalized Symmetric Metric Connection}

It is known from \cite{bib1} that the apapR manifold $(M,\varphi,\xi,\eta,g)$ is called para-Einstein-like with constants $(a,b,c)$ if its Ricci tensor $\mathrm{Ric}$ satisfies the following formula:
\begin{equation}
\label{11}
\mathrm{Ric}=a\ g+b\ \widetilde{g}+c\ \eta\otimes \eta
\end{equation}
where $a,b,c$ are constants. In particular, if $b=0$ and $b=c=0$, then the manifold is called an $\eta-$Einstein manifold and an Einstein manifold, respectively. If $a,b,c$ are functions on $M$, the manifold $M$ is called almost para-Einstein-like, almost $\eta-$Einstein-like or an almost Einstein manifold, respectively.

\begin{theorem}\label{thm1}
If the para-Sasaki-like manifold $(M,\varphi,\xi,\eta,g)$ is a para-Einstein-like manifold with constants $(a,b,c)$, then the apapR manifold $(M,\varphi,\xi,\eta,g)$ with the generalized symmetric metric connection is also a para-Einstein-like manifold with constants $(\lambda,\mu,\nu)$ determined by 
  \begin{equation}
\label{ }
\begin{array}{lll}
 \lambda &=&\beta^2-\beta+(1-2n)\alpha^2+a,  \\
  \mu    & =  &(2-2n)(\alpha\beta-\alpha)+\alpha+b \\
    \nu  &=& (2n-1)\alpha^2+(2n-2)\alpha\beta+(1-2n)\alpha+(2n+1)\beta-\beta^2+c
\end{array}
\end{equation}
\end{theorem}
\begin{proof}
Since the para-Sasaki-like manifold $(M,\varphi,\xi,\eta,g)$ is a para-Einstein-like manifold with constants $(a,b,c)$, (\ref{11}) is valid. Substituting (\ref{111}) in (\ref{11}), the theorem is proved by straightforward calculations.

\end{proof}

The concept of the para-Ricci-like soliton with potential $\xi$ is introduced in \cite{bib1}. An apapR manifold $(M,\varphi,\xi,\eta,g)$ admits a para-Ricci-like soliton with potential $\xi$ and constants $(\lambda,\mu,\nu)$ if the following identity is valid:
\begin{equation}
\label{51}
\frac{1}{2}\mathcal{L}_{\xi} g+\mathrm{Ric}+\lambda g+\mu \widetilde{g}+\nu \eta\otimes \eta=0,
\end{equation}
where $\mathcal{L}$ and $\mathrm{Ric}$ denote the Lie derivative and the Ricci tensor, respectively. The Lie derivative of $g$ along $v$ is defined by
\begin{equation}
\label{54}
(\mathcal{L}_{v} g)(x,y)=g(\nabla_xv,y)+g(x,\nabla_y v).
\end{equation}
If $\mu=0$ or $\mu=\nu=0$, then $M$ admits an $\eta-$Ricci soliton or a Ricci soliton, respectively. If $\lambda,\mu,\nu$ are functions on $M$, then the soliton is said to be almost para-Ricci-like soliton, almost $\eta-$Ricci soliton or almost Ricci soliton.
Moreover, this notion can be generalized for any potential. We say that  $(M,\varphi,\xi,\eta,g)$ admits a para-Ricci-like soliton with any potential vector field $v$ if the following condition is valid for a triplet of constants $(\lambda,\mu,\nu)$:
\begin{equation}
\label{53}
\frac{1}{2}\mathcal{L}_{v} g+\mathrm{Ric}+\lambda g+\mu \widetilde{g}+\nu \eta\otimes \eta=0.
\end{equation}
For a para-Sasaki-like manifold $(M,\varphi,\xi,\eta,g)$, we obtain 
\begin{equation}
\label{59}
\mathcal{L}_{\xi} g(x,y)=g(\nabla_x \xi,y)+g(x,\nabla_y \xi)=2g(x,\phi y).
\end{equation}
By using (\ref{54}) and (\ref{59}) we get 
\begin{equation}
\label{52}
\overline{\mathcal{L}}_{\xi} g(x,y)=\mathcal{L}_{\xi}g(x,y)-2\alpha g(\varphi x,\varphi y)-2\beta g(\varphi x,y).
\end{equation}

\begin{theorem}\label{thm2}
If the para-Sasaki-like manifold $(M,\varphi,\xi,\eta,g)$ admits a para-Ricci-like soliton with potential $\xi$ and constants $(a,b,c)$, then $(M,\varphi,\xi,\eta,g)$ with the generalized symmetric metric conneciton $\overline{\nabla}$ also admits a para-Ricci-like soliton with constants $(\lambda,\mu,\nu)$ determined by 
  \begin{equation}
\label{ }
\begin{array}{lll}
 \lambda &=& (2n-1)\alpha^2-\beta^2+\alpha+\beta+a \\
  \mu    & =  &2(n-1)\alpha\beta+(1-2n)\alpha+\beta+b \\
    \nu  &=& \beta^2+(1-2n)\alpha^2+(\alpha-\alpha\beta)(2n-2)-(2n+2)\beta+c
\end{array}
\end{equation}
\end{theorem}

\begin{proof}
Let the para-Sasaki-like manifold $(M,\varphi,\xi,\eta,g)$ admit a para-Ricci-like soliton with potential $\xi$ and constants $(a,b,c)$. Then, the relation (\ref{51}) is valid. If (\ref{52}) and (\ref{111}) are substituted in (\ref{51}), we infer the assertion.

\end{proof}

\begin{theorem}
Let $(M,\varphi,\xi,\eta,g)$ be a para-Sasaki-like Riemannian manifold of dimension $(2n+1)$ and let it admit a para-Ricci-like soliton with constants $(a,b,c)$ whose potential vector field $v$ satisfies the condition $v=k\xi$, i.e., it is pointwise collinear with the Reeb vector field $\xi$, where $k$ is a differentiable function on $M$. Then, $M$ also admits a para-Ricci-like soliton with respect to the generalized symmetric metric conneciton $\overline{\nabla}$ with constants $(\lambda,\mu,\nu)$ determined by 
\begin{equation}
\label{ }
\begin{array}{l}
 \lambda=a+\alpha b-\beta^2+\beta+(2n-1)\alpha^2         \\
\mu=b+\beta b+(\alpha-\alpha\beta)(2n-2)      \\
\nu=\alpha b-(2n+1)\beta+\beta^2+(1-2n)\alpha^2+c-\beta b+(\alpha\beta-\alpha)(2n-2)         
\end{array}
\end{equation}
\end{theorem}
\begin{proof}
Bearing in mind (\ref{15}) and  (\ref{54}), we compute Lie derivative of $g$ along $v=k\xi$ corresponding the generalized symmetric metric conneciton $\overline{\nabla}$ as follows
\begin{equation}
\label{ }
\overline{\mathcal{L}}_{v} g(x,y)=\mathcal{L}_{v}g(x,y)+2\alpha k\eta(x)\eta(y)-2\alpha k g(x,y)-2\beta g(x,\varphi y).
\end{equation}
Withdrawing $\mathcal{L}_{v}g(x,y)$ in the latter equality and using (\ref{111}) and (\ref{53}), we complete the proof.
\end{proof}

Note that if $(M,\varphi,\xi,\eta,g)$ admits a para-Ricci-like soliton with any potential $v$, it may not admit a para-Ricci-like soliton with respect to a generalized symmetric metric connection.

\section{Examples}
In this section, we construct two examples of the para-Sasaki-like manifold admitting a generalized symmetric metric connection and after that validate our results.

\subsection{Example 1}
In \cite{bib3}, a $3-$dimensional Lie group $L$ is considered with a global basis $\{\xi=e_0,e_1,e_2\}$ of the left invariant vector fields on $L$ such that the commutators of its associated Lie algebra are determined as follows
\begin{equation}
\label{eq33}
[e_0,e_1]=-e_2,\hspace{.5cm} [e_0,e_2]=-e_1, \hspace{.5cm} [e_1,e_2]=0.
\end{equation}
$L$ is endowed with an apapR structure $( \varphi, \xi, \eta, g)$ by
\begin{equation}
\label{eq34}
\begin{array}{ll}
 \varphi(e_0)=0,  \hspace{.5cm} \varphi(e_1)=e_2,\hspace{.5cm} \varphi(e_2)=e_1, \hspace{.5cm} \eta(e_0)=1,\\
    g(e_0,e_0)=g(e_1,e_1)=g(e_2,e_2) =1,   \\
   g(e_i,e_j)=0 \text{ for } i,j\in \{0,1,2\}, \ \ \ i\neq j.
\end{array}
\end{equation}
In the same paper, it is obtained that the basic components of the Levi-Civita connection $\nabla$, $R_{ijkl}=R(e_i,e_j,e_k,e_l)$ and $\text{Ric}_{ij}=\text{Ric}(e_i,e_j)$ are calculated by
\begin{equation}
\label{eq35}
\begin{array}{l}
\nabla_{e_1}e_0=e_2,\hspace{.5cm} \nabla_{e_2}e_0=e_1,\hspace{.5cm} \nabla_{e_1}e_2=\nabla_{e_2}e_1=-e_0;\\
R_{1221}=-R_{1001}=-R_{2002}=1, \hspace{.5cm} \text{ Ric}_{00}=-2.
\end{array}
\end{equation}
It is proved that the constructed manifold $L$ is a para-Einstein-like para-Sasaki-like Riemannian manifold with constants $(a,b,c)=(0,0,-2)$. Moreover,  $L$ is a para-Ricci-like soliton with constants $(0,-1,3)$.
The only non-zero component of  $(\mathcal{L}_{\xi} g)(e_i,e_j)=(\mathcal{L}_{\xi} g)_{ij}$ is given by $(\mathcal{L}_{\xi} g)_{12}=2$.

Now, let us consider the generalized symmetric metric connection $\overline{\nabla}$ on $L$. Taking into account (\ref{15}) and (\ref{eq35}), we compute the non-zero components of $\overline{\nabla}$ as follows
\begin{equation}
\label{eq36}
\begin{array}{l}
\overline{\nabla}_{e_1}e_0=(1-\beta)e_2-\alpha e_1,\hspace{.5cm} \overline{\nabla}_{e_2}e_0=(1-\beta)e_1-\alpha e_2,\\
\overline{\nabla}_{e_1}e_2=\overline{\nabla}_{e_2}e_1=(\beta-1)e_0, \hspace{.5cm} \overline{\nabla}_{e_1}e_1=\overline{\nabla}_{e_2}e_2=\alpha e_0.
\end{array}
\end{equation}
By virtue of (\ref{8}) and (\ref{eq36}), the components $\overline{R}_{ijkl}=\overline{R}(e_i,e_j,e_k,e_l)=g(\overline{R}(e_i,e_j)e_k,e_l)$ excluding the well-known their symmetries and antisymmetries are calculated  as follows:
\begin{equation}
\label{eq37}
\begin{array}{l}
\overline{R}_{0110}=-\overline{R}_{0202}=(\beta-1),\\
\overline{R}_{0102}=-\alpha\\
\overline{R}_{1212}=\alpha^2-(\beta-1)^2
\end{array}
\end{equation}
Using above relations, we compute the nonzero components of the Ricci tensor as follows:
\begin{equation}
\label{eq38}
\begin{array}{lll}
\overline{\mathrm{Ric}}_{00}=2(\beta-1),  \  \  \ &  \overline{\mathrm{Ric}}_{11}=\overline{\mathrm{Ric}}_{22}=\beta(\beta-1)-\alpha^2, \  \ \  &  \overline{\mathrm{Ric}}_{12}=\alpha.\\
\end{array}
\end{equation}
Taking into account (\ref{eq38}), we obtain $\overline{\text{scal}}=2(\beta^2-\alpha^2-1)$.

Therefore, by virtue of (\ref{eq38}), it can be easily seen that the manifold $(L,\phi,\xi,\eta,g)$ admits a para-Einstein-like manifold with regard to the generalized symmetric metric connection $\overline{\nabla}$ with constants $(\lambda,\mu,\nu)$ determined by
\begin{equation}
\label{eq39}
\begin{array}{lll}
\lambda=\beta(\beta-1)-\alpha^2,  \  \  \ &  \mu=\alpha, \  \ \  & \nu=\alpha(\alpha-1)-(\beta-1)(\beta-2).\\
\end{array}
\end{equation}
The non-zero components of $\mathcal{L}_{\xi}g$ are 
determined by the following equalities:
\begin{equation}
\label{eq55}
\begin{array}{lll}
(\overline{\mathcal{L}}_{\xi} g)_{11}=(\overline{\mathcal{L}}_{\xi} g)_{22}=-2 \alpha,  \  \   & (\overline{\mathcal{L}}_{\xi} g)_{12}=2(1-\beta).\\
\end{array}
\end{equation}
Therefore, using (\ref{eq38}) and (\ref{eq55}), we establish that the constructed manifold admits a paraRicci-like soliton with regard to the generalized symmetric metric connection $\overline{\nabla}$ with the following constants:
\begin{equation}
\label{eq41}
\begin{array}{lll}
\lambda=\alpha^2+\alpha+\beta-\beta^2,  \  \  \ &  \mu=\beta-\alpha-1, \  \ \  & \nu=(\beta-1)(\beta+3)-\alpha^2.\\
\end{array}
\end{equation}
In conclusion, the constructed $3-$dimensional example of a para-Sasaki-like Riemannian manifold $(L,\phi,\xi,\eta,g)$ with the generalized symmetric metric connection with the results in (\ref{eq39}) and (\ref{eq41}) support the proven assertions in Theorem \ref{thm1} and \ref{thm2}.

\subsection{Example 2}
Let us consider as an explicit example a $5-$dimensional Para-Sasaki-like manifold $G$ given in \cite{Ivanov}. Then, the Lie group $G$ has a basis of left-invariant vector fields $\{e_0,\ldots,e_4\}$ with corresponding Lie algebra determined as follows
\begin{equation*}
\begin{array}{l}
 [e_0,e_1]=pe_2-e_3+qe_4,  \ \  \   [e_0,e_2]=-pe_1-qe_3-e_4, \\
\end{array}
\end{equation*}
\begin{equation*}
\begin{array}{l}
 [e_0,e_3]=-e_1+qe_2+pe_4,  \ \  \  [e_0,e_4]=-qe_1-e_2-pe_3,\\
\end{array}
\end{equation*}
where $p,q\in \mathbb{R}$. The Lie group $G$ is endowed with an almost paracontact almost paracomplex structure $(\varphi,\xi,\eta,g)$ as follows:
\begin{equation}
\begin{array}{l}
  \xi=e_0, \ \  \ \varphi e_1=e_3, \ \  \ \varphi e_2=e_4,\ \  \ \varphi e_3=e_1,\ \  \ \varphi _4=e_2,     \\
   g(e_i,e_i)=1, \ \  \ g(e_i,e_j)=0, \ \  \ i,j\in \{0,1,\ldots,4\}, \ i\neq j.
\end{array}
\end{equation}
The non-zero components of the Levi-Civita connection are given in the following way:
\begin{equation}
\begin{array}{l}
 \nabla_{e_0}e_1=p e_2+qe_4, \ \  \ \nabla_{e_1} e_0=e_3, \ \  \ \nabla_{e_0} e_2=-pe_1-qe_3,\ \  \ \nabla_{e_2} e_0=e_4,     \\
 \nabla_{e_0}e_3=q e_2+p e_4, \ \  \   \nabla_{e_3} e_0=e_1, \ \  \ \nabla_{e_0} e_4=-qe_1-pe_3,\ \  \ \nabla_{e_4} e_0=e_2,     \\
  \nabla_{e_1}e_3= \nabla_{e_2} e_4=\nabla_{e_3} e_1= \nabla_{e_4} e_2=-e_0.     \\
 \end{array}
\end{equation}
In \cite{bib2}, it is proved that $(G,\varphi,\xi,\eta,g)$ is a para-Sasaki-like manifold and \\ $\eta-$Einstein with constants $(a,b,c)=(0,0,-4)$. Moreover, it admits a para-Ricci-like soliton with potential $\xi$ constants $(0,-1,5)$. The only non-zero component of the Ricci tensor $\mathrm{Ric}$ is $\mathrm{Ric}_{00}=-4$. The non-zero of the components $R_{ijkl}=R(e_i,e_j,e_k,e_l)=g(R(e_i,e_j)e_k,e_l)$ are determined by the following equalities with the well-known their symmetries and antisymmetries:
\begin{equation}\label{eq40}
\begin{array}{l}
R_{0110}=R_{0220}=R_{0330}=R_{0440}=-1, \\
R_{1234}=R_{1432}=R_{1331}=R_{2442}=1.
\end{array}
\end{equation}
The non-zero components of $\mathcal{L}_{\xi}g$ are 
determined by the following equalities:
\begin{equation*}
\begin{array}{l}
( \mathcal{L}_{\xi} g)_{13}=( \mathcal{L}_{\xi} g)_{24}=( \mathcal{L}_{\xi} g)_{31}=( \mathcal{L}_{\xi} g)_{42}=2. \\
\end{array}
\end{equation*}
The non-zero components of the generalized symmetric metric connection $\overline{ \nabla}$ are given in the following way:
\begin{equation}
\begin{array}{l}
 \overline{\nabla}_{e_0}e_1=\nabla_{e_0}e_1=p e_2+qe_4, \ \   \  \ \ \ \overline{\nabla}_{e_1} e_0=(1-\beta)e_3-\alpha e_1,   \\
 \overline{\nabla}_{e_0}e_2=\nabla_{e_0}e_2=-p e_1-qe_3, \ \   \  \ \overline{\nabla}_{e_2} e_0=(1-\beta)e_4-\alpha e_2,  \\
 \overline{\nabla}_{e_0}e_3=\nabla_{e_0}e_3=q e_2+pe_4, \ \ \ \ \   \  \ \overline{\nabla}_{e_3} e_0=(1-\beta)e_1-\alpha e_3,   \\
 \overline{\nabla}_{e_0}e_4=\nabla_{e_0}e_4=-q e_1-pe_3, \ \ \  \  \ \overline{\nabla}_{e_4} e_0=(1-\beta)e_2-\alpha e_4,   \\
 \overline{\nabla}_{e_1}e_3= \overline{\nabla}_{e_3}e_1= \overline{\nabla}_{e_2}e_4= \overline{\nabla}_{e_4}e_2=(\beta-1)e_0,  \\
  \overline{\nabla}_{e_i}e_i=\alpha e_0, \ \ i\neq 0.
 \end{array}
\end{equation}
Taking into account (\ref{8}) and (\ref{eq40}), we compute the components $\overline{R}_{ijkl}=\overline{R}(e_i,e_j,e_k,e_l)=g(\overline{R}(e_i,e_j)e_k,e_l)$ excluding the well-known their symmetries and antisymmetries as follows:
\begin{equation}
\begin{array}{l}
\overline R_{0101}=-\overline R_{0440}=\overline R_{0202}=-\overline R_{0330}=1-\beta,  \\
\overline R_{0103}=\overline R_{0402}=-\alpha,  \\
\overline R_{1412}=-\overline R_{3423}=-\overline R_{3441}=-\overline R_{2312}=\alpha(\beta-1),  \\
\overline R_{3434}=\overline R_{2323}=\overline R_{1414}=\overline R_{1212}=\alpha^2,\\
\overline R_{1313}=\overline R_{2424}=\alpha^2-(\beta-1)^2,\\
\overline R_{1234}=-\overline R_{2314}=(\beta-1)^2.
 \end{array}
\end{equation}
The nonzero components of the Ricci tensor $\overline {Ric}$ are calculated as follows:
\begin{equation}
\label{eq60}
\begin{array}{l}
\overline {Ric}_{00}=4(\beta-1),         \\
 \overline {Ric}_{11}=\overline {Ric}_{22}=\overline {Ric}_{33}=\overline {Ric}_{44}=\beta^2-\beta-3\alpha^2,        \\
 \overline {Ric}_{13}=\overline {Ric}_{24}=3\alpha-2\alpha\beta.        
\end{array}
\end{equation}
Hence, by using above relations the scalar curvature is obtained by $$\overline {scal}=4(\beta^2-3\alpha^2-1).$$
Moreover, the non-zero components of the Lie derivative of $g$ along $\xi$ with  connection $\overline{ \nabla}$  are given by
\begin{equation}
\label{eq61}
\begin{array}{l}
(\overline {\mathcal{L}}_{\xi} g)_{11}=(\overline {\mathcal{L}}_{\xi} g)_{22}=(\overline {\mathcal{L}}_{\xi} g)_{33}=(\overline {\mathcal{L}}_\xi g)_{44}=-2\alpha,          \\
(\overline {\mathcal{L}}_\xi g)_{13}=(\overline {\mathcal{L}}_\xi g)_{24}=2(1-\beta)         
\end{array}
\end{equation}
By virtue of (\ref{eq60}) and ({\ref{eq61}}),  the considered manifold  $(G,\varphi,\xi,\eta,g)$ is para-Einstein-like manifold with constants $(\lambda,\mu,\nu)$ given by
\begin{equation}
\label{eq62}
\begin{array}{l}
\lambda=\beta^2-\beta-3\alpha^2,          \\
 \mu=-2\alpha\beta+3\alpha,       \\
\nu=3\alpha^2+2\alpha\beta-3\alpha+5\beta-\beta^2-4.         
\end{array}
\end{equation}
and it admits a para-Ricci-like soliton with potential $\xi$ with constants $(\lambda,\mu,\nu)$ determined by
\begin{equation}
\label{eq63}
\begin{array}{l}
\lambda=3\alpha^2+\alpha-\beta^2+\beta,          \\
\mu=2\alpha\beta+\beta-3\alpha-1,    \\
\nu=\beta^2-3\alpha^2-2\alpha\beta-6\beta+2\alpha+5.     
\end{array}
\end{equation}
Consequently, the obtained manifold $(G,\phi,\xi,\eta,g)$ with the results in (\ref{eq62}) and (\ref{eq63}) support the proven assertions in Theorem \ref{thm1} and \ref{thm2}.

\begin{acknowledgement}
The first author was supported by Project 22ADP356 of Scientific Research Projects Commission, Eskişehir Technical University.

\end{acknowledgement}

\end{document}